\documentclass[12pt,reqno,a4paper]{amsart}
\usepackage{amsmath}
\usepackage{amssymb}
\usepackage{stmaryrd}
\usepackage{amsthm}

\usepackage[applemac]{inputenc}

\advance\voffset    by -1  cm
\advance\hoffset    by -1.5cm
\advance\textwidth  by  3  cm
\advance\textheight by  1  cm
\sffamily

\usepackage{pstricks}
\usepackage{multido}

\usepackage{graphicx}
\usepackage{color}

\newcommand{\showgrid}{}

\newcommand{\gridon}{\renewcommand{\showgrid}{\psset{subgriddiv=1,griddots=10,gridlabels=6pt}\psgrid}}

\gridon 

\newgray{hellgrau}{0.9}
\newrgbcolor{dummycolor}{0.75 0.85 0.0}

\long\def\comment#1{%
\psshadowbox[%
fillstyle=solid,fillcolor=dummycolor,linewidth=0.05%
]{\hbox to 12.70cm {\vbox{\strut {\bf Kommentar:} #1\hfill}}}%
}

\newif\ifenglish

\englishtrue

\newif\ifvariant
\variantfalse

\def\bit{\begin{itemize}}
\def\eit{\end{itemize}}
\def\beq{\begin{equation}}
\def\eeq{\end{equation}}

\def\EM#1{{\em #1}}

\def\figref#1{\ifenglish Figure\else Abbildung\fi~\ref{#1}}
\def\secref#1{\ifenglish Section\else Abschnitt\fi~\ref{#1}}

\def\N{{\mathbb N}}

\def\1{{\mathbf 1}}
\def\0{{\mathbf 0}}

\def\defeq{:=}

\def\of#1{\left(#1\right)}

\def\setof#1{\left\{#1\right\}}

\def\pas#1{\left(#1\right)}
\def\brk#1{\left[#1\right]}

\def\sgn{\operatorname{sgn}} 

\def\range#1{\left[#1\right]}  
  
\def\disjuni{\,\dot{\cup}\,}

\def\card#1{\left\vert{#1}\right\vert}

\def\weight{{\mathbf\omega}}

\def\symm{\mathfrak{S}}

\newif\ifenglish
\englishtrue

\newtheorem{thm}{\ifenglish Theorem\else Satz\fi}

\newtheorem{lem}{Lemma}

\newtheorem{cor}{\ifenglish Corollary\else Korollar\fi}

\theoremstyle{remark}

\def\secref#1{Section~\ref{#1}}

\def\weight{\omega}

\def\concat{{\scriptstyle\circ}}

\def\numof#1{\#\of{#1}}                

\def\symm{\mathfrak{S}}

\def\disjuni{\,\dot{\cup}\,}

\def\and{\wedge}

\def\myrange#1{\setof{1,\dots,k}}

\def\cop{{\sl\bf cp}}
\def\coc{{\sl\bf cocp}}
\def\pointconf{configuration of (starting/ending) points (short: \cop){\gdef\pointconf{\cop}}}
\def\pointarr{circular orientation of coloured (starting/ending) points (short: \coc){\gdef\pointarr{\coc}}}

\def\matminor#1#2#3{{#1}_{#2,#3}}
\def\matcominor#1#2#3{{#1}_{\overline{#2},\overline{#3}}}

\def\domain{\operatorname{domain}}
\def\image{\operatorname{image}}

\def\speciesname#1{\mathtt{#1}}

\def\signsumset#1#2{\sgn\of{#1\!\trianglelefteq\!#2}}

\def\concat{\star}
\def\complies{\subseteq}
\def\scheme#1{\brk{#1}}
\def\compat#1#2{{\mathbf #1}_{\brk{#2}}}

\def\complementoffirst#1#2{{\mathbf #1}_{#2}}

\newcount\labelcount

\newif\iflongversion
\longversionfalse

\begin{document}

\bibliographystyle{plain}

\title{A combinatorial proof for Cayley's identity}

\begin{abstract}
In \cite{Caracciolo:2013}, Caracciolo, Sokal and Sportiello presented, {\it inter alia},
an algebraic/combinatorial proof for Cayley's identity. The purpose of the present
paper is to give a ``purely
combinatorial'' proof for this identity; i.e., a proof involving only combinatorial arguments together
with a generalization of Laplace's Theorem \cite[section 148]{Muir:1933}, for which
a ``purely combinatorial'' proof is given in \cite[proof of Theorem 6]{Fulmek:2012}.
\end{abstract}

\author{Markus Fulmek}
\address{Fakult\"at f\"ur Mathematik, 
Oskar-Morgenstern-Platz 1, A-1090 Wien, Austria}
\email{{\tt Markus.Fulmek@Univie.Ac.At}\newline\leavevmode\indent
{\it WWW}: {\tt http://www.mat.univie.ac.at/\~{}mfulmek}
}

\date{\today}
\thanks{
Research supported by the National Research Network ``Analytic
Combinatorics and Probabilistic Number Theory'', funded by the
Austrian Science Foundation. 
}

\maketitle

\section{Introduction}
\label{sec:intro}

For $n\in\N$, denote by $\range{n}$ the set $\setof{1,2,\dots,n}$ and let $X=X_n
=\pas{x_{i,j}}_{\pas{i,j}\in\range{n}\times\range{n}}$
be an $n \times n$ matrix of indeterminates. For $I\subseteq\range{n}$ and $J\subseteq\range{n}$, we
denote
\bit
\item the \EM{minor} of $X$ corresponding to the rows $i\in I$ and the columns $j\in J$ by
$\matminor{X}{I}{J}$,
\item the cominor of $\matminor X I J$ 
(which corresponds to the rows $i\not\in I$ and the columns
$j\not\in J$) by $\matcominor{X}{I}{J}$.
\eit

Let $M=\setof{x_1\leq x_2\leq\dots\leq x_m}$ be a finite ordered
set, and let $S=\setof{x_{i_1},\dots, x_{i_k}}\subseteq M$ be a subset of $M$. We define
$$\signsumset{S}{M}\defeq\pas{-1}^{\sum_{j=1}^k i_j}.$$


As pointed out in \cite[Section 2.6]{Caracciolo:2013}, the following identity is
conventionally but erroneously attributed to 
Cayley. (Muir
\cite[vol. 4, p. 479]{Muir:1906} attributes this identity to Vivanti \cite{Vivanti:1890}.)
\begin{thm}[Cayley's Identity]
\label{thm:cayley}
Consider $X = \pas{x_{i,j}}_{\pas{i,j}\in\range{n}\times\range{n}}$, and let
$\partial = \pas{\frac{\partial}{\partial x_{i,j}}}$ be the corresponding
$n \times n$ matrix of partial derivatives\footnote{$\partial$ is also known as \EM{Cayley's $\Omega$--process}.}.
Let $I,J\subseteq\range{n}$ with $\card{I}=\card{J}=k$. Then we have for $s\in\N$:
\begin{multline}
\label{eq:cayley}
\det\of{\matminor{\partial}{I}{J}}\pas{\det\of X}^s = \\
s\cdot\pas{s+1}\cdots\pas{s+k-1}\cdot\pas{\det\of X}^{s-1}\cdot
\signsumset{I}{\range n}\cdot\signsumset{J}{\range n}
\cdot\det\of{\matcominor{X}{I}{J}}.
\end{multline}
\end{thm}
By the alternating property of the determinant, Cayley's Identity is
in fact equivalent to the following special case of \eqref{eq:cayley}.

\begin{cor}[Vivanti's Theorem]
\label{cor:vivanti}
Specialize $I=J=\range{k}$ for some $k\leq n$ in Theorem~\ref{thm:cayley}.
Then we have for $s\in\N$:
\begin{multline}
\label{eq:vivanti}
\det\of{\matminor{\partial}{\range{k}}{\range{k}}}\pas{\det\of X}^s
=
s\cdot\pas{s+1}\cdots\pas{s+k-1}\cdot\pas{\det\of X}^{s-1}\cdot\det\of{\matcominor{X}{\range k}{\range k}}.
\end{multline}
\end{cor}

\section{Combinatorial proof of Vivanti's Theorem}
We may \EM{view} the 
determinant of $X$
as the \EM{generating function} of all permutations $\pi$ in $\symm_n$,
where the (signed) weight of a permutation $\pi$ is given as
$\weight\of\pi\defeq\sgn\pi\cdot\prod_{i=1}^nx_{i,\pi\of i}$:
$$
\det\of X =\sum_{\pi\in \symm_n}\weight\of\pi. 
$$
\subsection{View permutations as perfect matchings}
For our considerations, it is convenient to \EM{view} a permutation
$\pi\in\symm_n$ as a \EM{perfect matching} $m_\pi$ of the complete bipartite graph $K_{n,n}$,
where the vertices consist of two copies of $\range{n}$ which are arranged in their natural order;
see \figref{fig:permmatch} for an illustration of this simple idea.
It is easy to see that the edges of such perfect matching can be drawn
in a way such that all \EM{intersections} are of precisely two (and not more)
edges, and that the number of these intersections equals the number of
\EM{inversions} of $\pi$, whence the sign  of $\pi$ is
$$
\sgn\of\pi=\pas{-1}^{\numof{\text{intersections in }m_\pi}}.
$$
This simple visualization of permutations and their inversions is already used in \cite[\S 15, p.32]{Aitken:1956}: We call it the \EM{permutation diagram}.
So assigning weight $x_{i,j}$ to the edge pointing from $i$ to $j$ and defining the
weight $\weight\of{m_\pi}$ of the permutation diagram $m_\pi$ 
to be the product of the edges belonging to $m_\pi$, we may write
$$
\weight\of\pi=\pas{-1}^{\numof{\text{intersections in }m_\pi}}\cdot\weight\of{m_\pi}.
$$
\begin{figure}
\caption{View the permutation $\pi=\pas{2 5 1 4 3}$ as the corresponding perfect matching
$m_\pi$ in the complete bipartite graph $K_{5,5}$. The intersections of edges
are indicated by small circles; they correspond bijectively to $\pi$'s inversions:
$$\numof{\text{inversions of }\pi}=\card{\setof{(1,3),(2,4),(2,5),(4,5)}}=4.$$
Assigning weight $x_{i,j}$ to the edge pointing from $i$ to $j$ gives the
contribution of the permutation $\pi$ to the determinant of $X_5$:
$$
\weight\of\pi = \pas{-1}^4\cdot x_{1,2}\cdot x_{2,5}\cdot x_{3,1}\cdot x_{4,4}\cdot x_{5,3}.
$$
}
\label{fig:permmatch}
\input graphics/permmatch
\end{figure}

Given this view, the combinatorial interpretation of the $s$-th power of the determinant $\det\of X$ is
obvious: It is the generating function of all $s$-tuples $m=\pas{m_{\pi_1},\dots,m_{\pi_s}}$
of permutation diagrams, where the (signed) weight of such $s$-tuple $m$ is given as
$$
\weight\of m=\prod_{i=1}^s\pas{-1}^{\numof{\text{intersections in }m_{\pi_i}}}\cdot\weight\of{m_{\pi_i}}.
$$
(See \figref{fig:det2s} for an illustration.)
\begin{figure}
\caption{For $n=5$, the picture shows a typical object of weight
$$
\!\!\!\!\!\!\!\!\!\!\!\!\!\!\!\!\!\!\!\!\!
-
x_{1,1}
x_{1,2}^2
x_{1,4}
x_{2,3}
x_{2,4}
x_{2,5}^2
x_{3,1}
x_{3,3}
x_{3,4}
x_{3,5}
x_{4,1}
x_{4,3}
x_{4,4}
x_{4,5}
x_{5,1}
x_{5,2}^2
x_{5,3},
$$
which is counted by the generating function $\det\of{X}^4$.
(The edge connecting lower vertex $3$ to upper vertex $3$ in the $4$--th (right--most) matching
is drawn as zigzag-line, just to avoid
intersections of more than two edges in a single point.)
}
\label{fig:det2s}
\input graphics/det2s
\end{figure}

\subsection{Action of the determinant of partial derivatives}
Next we need to describe combinatorially the \EM{action} of the determinant $\det\of{\matminor{\partial}{\range{k}}{\range{k}}}$
of partial derivatives. Let $m=\pas{m_{\pi_1},\dots,m_{\pi_s}}$ be an $s$-tuple of permutation diagrams
counted in the generating function $\pas{\det\of X}^s$, and let $\tau\in\symm_k$:
Then the summand
$$
\partial_\tau\defeq\sgn\of\tau\cdot\prod_{i=1}^k\frac{\partial}{\partial x_{i,\tau\of i}}
$$
applied to $\weight\of m$ yields
$$
\sgn\of\tau\cdot\pas{\prod_{i=1}^k\frac{\partial}{\partial x_{i,\tau\of i}}}\weight\of m =
\sgn\of\tau\cdot c_{\tau,m}\cdot \frac{\weight\of m}{\prod_{i=1}^k{x_{i,\tau\of i}}},
$$
where $c_{\tau,m}$ is the number of ways to \EM{choose} the set of $k$ edges
$\setof{\pas{i\to\tau\of i}:i\in\range{k}}$ from \EM{all} the edges in $m$ (this number, of course,
might be zero). We may visualize the action of $\delta_\tau$ as ``erasing the edges constituting $\tau$
in $m$''; see \figref{fig:derivdet2s} for an illustration.

\begin{figure}
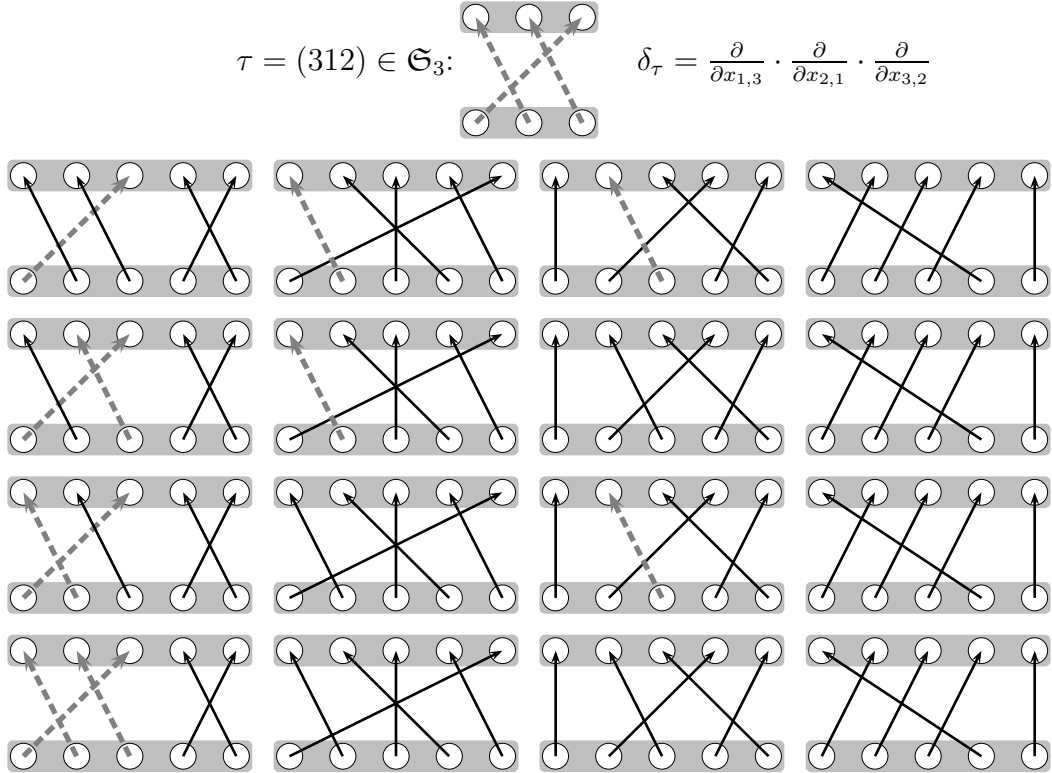

\caption{Let $n=5$, $s=4$ and $k=3$ in Corollary~\ref{cor:vivanti}. The picture shows the four possible ways of ``erasing'' the edges constituting
$\tau\in\symm_3$ from the $4$-tuple $\pas{m_{\pi_1},m_{\pi_2},m_{\pi_3},m_{\pi_4}}$ of
matchings, where $\pas{\pi_1,\pi_2,\pi_3,\pi_4}\in\symm_5^4$ is
$\pas{\pas{3 1 2 5 4},\pas{5 1 3 2 4},\pas{1 4 2 5 3},\pas{2 3 4 1 5}}$. The erased edges are shown as
grey dashed lines.
}
\label{fig:derivdet2s}
\psset{unit=0.7cm}
\pspicture(-0.5,-3.5)(20.5,12)
\rput[Bl](5,10){$\tau=\pas{3 1 2}\in\symm_3$:}
\rput[Bl](12.53,10){$\delta_\tau=
\frac{\partial}{\partial x_{1,3}}\cdot\frac{\partial}{\partial x_{2,1}}\cdot\frac{\partial}{\partial x_{3,2}}$}
\psset{linecolor=lightgray,framearc=0.3,fillstyle=solid,fillcolor=lightgray}
\psframe(9.2,8.7)(11.8,9.3)
\psframe(9.2,10.7)(11.8,11.3)
\multido{\nz=0+5}{4}{
	\rput(\nz,-3){\psframe(0.7,-0.3)(5.3,0.3)}
	\rput(\nz,-1){\psframe(0.7,-0.3)(5.3,0.3)}
}\multido{\nz=0+5}{4}{
	\rput(\nz,0){\psframe(0.7,-0.3)(5.3,0.3)}
	\rput(\nz,2){\psframe(0.7,-0.3)(5.3,0.3)}
}
\multido{\nz=0+5}{4}{
	\rput(\nz,3){\psframe(0.7,-0.3)(5.3,0.3)}
	\rput(\nz,5){\psframe(0.7,-0.3)(5.3,0.3)}
}
\multido{\nz=0+5}{4}{
	\rput(\nz,6){\psframe(0.7,-0.3)(5.3,0.3)}
	\rput(\nz,8){\psframe(0.7,-0.3)(5.3,0.3)}
}
\psset{linecolor=black,linewidth=0.01,fillstyle=solid,fillcolor=white}
\multido{\nx=9+1}{3}{
	\rput(\nx,0){\pscircle(0.5,9){0.25}}
	{\rput(\nx,0){\pscircle(0.5,11){0.25}}}
}
\rput(0,-3){\multido{\nz=0+5}{4}{
	\rput(\nz,0){
		\multido{\nx=1+1}{5}{\pscircle(\nx,0){0.25}}
		\multido{\nx=1+1}{5}{\pscircle(\nx,2){0.25}}
	}
}
}
\rput(0,0){\multido{\nz=0+5}{4}{
	\rput(\nz,0){
		\multido{\nx=1+1}{5}{\pscircle(\nx,0){0.25}}
		\multido{\nx=1+1}{5}{\pscircle(\nx,2){0.25}}
	}
}
}
\rput(0,3){\multido{\nz=0+5}{4}{
	\rput(\nz,0){
		\multido{\nx=1+1}{5}{\pscircle(\nx,0){0.25}}
		\multido{\nx=1+1}{5}{\pscircle(\nx,2){0.25}}
	}
}
}
\rput(0,6){\multido{\nz=0+5}{4}{
	\rput(\nz,0){
		\multido{\nx=1+1}{5}{\pscircle(\nx,0){0.25}}
		\multido{\nx=1+1}{5}{\pscircle(\nx,2){0.25}}
	}
}
}
\psset{linewidth=0.1,linecolor=gray,fillstyle=none,linestyle=dashed,dash=4pt 2pt}
\psline{->}(9.5,9)(11.5,11)
\psline{->}(10.5,9)(9.5,11)
\psline{->}(11.5,9)(10.5,11)
\psset{linewidth=0.05,linecolor=black,fillstyle=none,linestyle=solid}
\rput(0,-3){
\psline[linewidth=0.1,linecolor=gray,fillstyle=none,linestyle=dashed,dash=4pt 2pt]{->}(1,0)(3,2) 
\psline[linewidth=0.1,linecolor=gray,fillstyle=none,linestyle=dashed,dash=4pt 2pt]{->}(2,0)(1,2)
\psline[linewidth=0.1,linecolor=gray,fillstyle=none,linestyle=dashed,dash=4pt 2pt]{->}(3,0)(2,2)
\psline{->}(4,0)(5,2)
\psline{->}(5,0)(4,2)
\psline{->}(6,0)(10,2)
\psline{->}(7,0)(6,2)
\psline{->}(8,0)(8,2)
\psline{->}(9,0)(7,2)
\psline{->}(10,0)(9,2) 
\psline{->}(11,0)(11,2)
\psline{->}(12,0)(14,2)
\psline{->}(13,0)(12,2)
\psline{->}(14,0)(15,2)
\psline{->}(15,0)(13,2) 
\psline{->}(16,0)(17,2)
\psline{->}(17,0)(18,2)
\psline{->}(18,0)(19,2)
\psline{->}(19,0)(16,2)
\psline{->}(20,0)(20,2) 
}
\rput(0,0){
\psline[linewidth=0.1,linecolor=gray,fillstyle=none,linestyle=dashed,dash=4pt 2pt]{->}(1,0)(3,2) 
\psline[linewidth=0.1,linecolor=gray,fillstyle=none,linestyle=dashed,dash=4pt 2pt]{->}(2,0)(1,2)
\psline{->}(3,0)(2,2)
\psline{->}(4,0)(5,2)
\psline{->}(5,0)(4,2)
\psline{->}(6,0)(10,2)
\psline{->}(7,0)(6,2)
\psline{->}(8,0)(8,2)
\psline{->}(9,0)(7,2)
\psline{->}(10,0)(9,2) 
\psline{->}(11,0)(11,2)
\psline{->}(12,0)(14,2)
\psline[linewidth=0.1,linecolor=gray,fillstyle=none,linestyle=dashed,dash=4pt 2pt]{->}(13,0)(12,2)
\psline{->}(14,0)(15,2)
\psline{->}(15,0)(13,2) 
\psline{->}(16,0)(17,2)
\psline{->}(17,0)(18,2)
\psline{->}(18,0)(19,2)
\psline{->}(19,0)(16,2)
\psline{->}(20,0)(20,2) 
}
\rput(0,3){
\psline[linewidth=0.1,linecolor=gray,fillstyle=none,linestyle=dashed,dash=4pt 2pt]{->}(1,0)(3,2) 
\psline{->}(2,0)(1,2)
\psline[linewidth=0.1,linecolor=gray,fillstyle=none,linestyle=dashed,dash=4pt 2pt]{->}(3,0)(2,2)
\psline{->}(4,0)(5,2)
\psline{->}(5,0)(4,2)
\psline{->}(6,0)(10,2)
\psline[linewidth=0.1,linecolor=gray,fillstyle=none,linestyle=dashed,dash=4pt 2pt]{->}(7,0)(6,2)
\psline{->}(8,0)(8,2)
\psline{->}(9,0)(7,2)
\psline{->}(10,0)(9,2) 
\psline{->}(11,0)(11,2)
\psline{->}(12,0)(14,2)
\psline{->}(13,0)(12,2)
\psline{->}(14,0)(15,2)
\psline{->}(15,0)(13,2) 
\psline{->}(16,0)(17,2)
\psline{->}(17,0)(18,2)
\psline{->}(18,0)(19,2)
\psline{->}(19,0)(16,2)
\psline{->}(20,0)(20,2) 
}
\rput(0,6){
\psline[linewidth=0.1,linecolor=gray,fillstyle=none,linestyle=dashed,dash=4pt 2pt]{->}(1,0)(3,2) 
\psline{->}(2,0)(1,2)
\psline{->}(3,0)(2,2)
\psline{->}(4,0)(5,2)
\psline{->}(5,0)(4,2)
\psline{->}(6,0)(10,2)
\psline[linewidth=0.1,linecolor=gray,fillstyle=none,linestyle=dashed,dash=4pt 2pt]{->}(7,0)(6,2)
\psline{->}(8,0)(8,2)
\psline{->}(9,0)(7,2)
\psline{->}(10,0)(9,2) 
\psline{->}(11,0)(11,2)
\psline{->}(12,0)(14,2)
\psline[linewidth=0.1,linecolor=gray,fillstyle=none,linestyle=dashed,dash=4pt 2pt]{->}(13,0)(12,2)
\psline{->}(14,0)(15,2)
\psline{->}(15,0)(13,2) 
\psline{->}(16,0)(17,2)
\psline{->}(17,0)(18,2)
\psline{->}(18,0)(19,2)
\psline{->}(19,0)(16,2)
\psline{->}(20,0)(20,2) 
}
\endpspicture
 
\end{figure}

Hence we have:
\begin{equation}
\label{eq:partialaction}
\det\of{\matminor{\partial}{\range{k}}{\range{k}}}\pas{\det\of X}^s
=
\sum_{m\in\symm_n^s}\!\!\weight\of m
\sum_{\tau\in\symm_k}c_{\tau,m}\cdot\frac{\sgn\of\tau }{\prod_{i=1}^k{x_{i,\tau\of i}}}.
\end{equation}
\subsection{Double counting} For our purposes, it is convenient to interchange
the summation in \eqref{eq:partialaction}. This application of \EM{double counting} amounts
here to a simple change of view: Instead of counting the ways to \EM{choose} the set of edges corresponding
to $\tau$ from all the edges corresponding to some \EM{fixed} $s$-tuple $m$, we
fix $\tau$ and consider the set of $m$'s from which $\tau$' edges might be chosen. 
This will involve two considerations:
\bit
\item In how many ways can the edges corresponding to $\tau$ be \EM{distributed} on
	$s$ copies of the bipartite graph $K_{n,n}$?
\item For each such distribution, what is the set of compatible $s$-tuples of permutation diagrams?
\eit
For example, if $k=3$ and $s=4$ (as in \figref{fig:derivdet2s}), there clearly
\bit
\item is $1$ way to distribute the three edges on a \EM{single} copy of the $4$ bipartite
	graphs (see the fourth row of pictures in \figref{fig:derivdet2s}), and there are $4$ ways to
	choose such single copy,
\item are $3$ ways to distribute the three edges on \EM{precisely two} copies of the $4$ bipartite
	graphs (see the second and third row of pictures in \figref{fig:derivdet2s}), and there are $4\cdot 3$
	ways to choose such pair of copies (whose order is relevant),
\item is $1$ way to distribute the three edges on \EM{precisely three} copies of the $4$ bipartite
	graphs (see the first row of pictures in \figref{fig:derivdet2s}), and there are $4\cdot 3 \cdot 2$
	ways to choose such triple of copies (whose order is relevant).
\eit

\subsection{Partitioned permutations}
A distribution of the edges corresponding to $\tau\in\symm_k$ on $s$ copies of the
bipartite graph $K_{n,n}$ may be viewed (see \figref{fig:derivdet2s})
\bit
\item as an $s$-tuple of \EM{partial matchings} (some of
which may be empty) of $K_{k,k}$
\item such that the union of these $s$
partial matchings gives the \EM{perfect matching} $m_\tau$ of  $K_{k,k}$.
\eit
Clearly, to each of such partial matching corresponds a \EM{partial permutation} $\tau_i$, which
we may write in two-line notation as follows:
\bit
\item the lower line shows the \EM{domain} of $\tau_i$ in its natural order,
\item the upper line shows the \EM{image} of $\tau_i$,
\item the \EM{ordering} of the upper line represents the permutation $\tau_i$.
\eit
We say that each of these $\tau_i$ is a \EM{partial permutation} of $\tau$, and that
$\tau$ is a \EM{partitioned permutation}. We write in short:
$$
\tau =\tau_1\concat\tau_2\concat\dots\concat\tau_s.
$$
For example, the rows of pictures  in \figref{fig:derivdet2s} correspond to the 
partitioned permutations
(written in the aforementioned two-line notation)
\bit
\item ${\binom{3}{1}\concat\binom{1}{2}\concat\binom{2}{3}\concat\binom{}{}}$ for the first row,
\item ${\binom{3\, 2}{1\, 3}\concat\binom{1}{2}\concat\binom{}{}\concat\binom{}{}}$ for the second row,
\item ${\binom{3\, 1}{1\, 2}\concat\binom{}{}\concat\binom{2}{3}\concat\binom{}{}}$ for the third row,
\item ${\binom{3\, 1\, 2}{1\, 2\, 3}\concat\binom{}{}\concat\binom{}{}\concat\binom{}{}}$ for the fourth row.
\eit

\subsection{Equivalence relation for partitioned permutations}
\label{sec:equivalence}
For any partitioned permutation $\tau=\tau_1\concat\tau_2\concat\dots\concat\tau_s$,
consider the $s$-tuple of the \EM{upper rows} (in the aforementioned two-line notation) \EM{only}:
We call this $s$-tuple of \EM{permutation words} the \EM{partition scheme} of $\tau$ and denote it
by $\scheme{\tau}$. We say that $\tau=\tau_1\concat\tau_2\concat\dots\concat\tau_s$
\EM{complies} to its partition scheme $\scheme{\tau}=\scheme{\tau_1\concat\tau_2\concat\dots\concat\tau_s}$
and denote this by ${\tau}\complies{\scheme{\tau_1\concat\tau_2\concat\dots\concat\tau_s}}$.

Now consider the following equivalence relation on the set of partitioned permutations:
$$
\mu={\mu_1\concat\dots\concat\mu_s}\sim\nu={\nu_1\concat\dots\concat\nu_s}:\iff\scheme{\mu}=\scheme{\nu}.
$$
By definition, the corresponding equivalence classes are \EM{indexed} by a partition scheme,
and $\mu=\mu_1\concat\mu_2\concat\dots\concat\mu_s$ belongs to the equivalence class
of  $\tau=\tau_1\concat\tau_2\concat\dots\concat\tau_s$ iff $\mu\complies\scheme{\tau}$.
(For $s>1$, a partitioned permutation $\tau$ is \EM{not}
uniquely determined by $\scheme{\tau}$.)

It is straightforward to compute the \EM{number} of these equivalence classes: In the language of
\EM{combinatorial species} (see, for instance, \cite{Bergeron-Labelle-Leroux:1998}) the $s$-tuples
of permutation
words indexing these
classes correspond bijectively to the (labelled) species $\pas{\speciesname{Permutations}}^s$, and since the
exponential generating function of $\speciesname{Permutations}$ is
$$
\sum_{n=0}^\infty n!\cdot\frac{z^n}{n!}= \frac1{1-z},
$$
the
exponential generating function of $\pas{\speciesname{Permutations}}^s$ is simply
$$
\pas{\frac1{1-z}}^s = \pas{1-z}^{-s}=\sum_{k=0}^\infty\binom{-s}{k}\pas{-1}^k z^k=
\sum_{k=0}^\infty s\cdot\pas{s+1}\cdots\pas{s+k-1} \frac{z^k}{k!}.
$$
So the number of these equivalence classes is $s\cdot\pas{s+1}\cdots\pas{s+k-1}$, which
is precisely the factor in \eqref{eq:vivanti}. Our proof will be complete if we manage
to show that the generating
functions of \EM{each} of these equivalence classes are the \EM{same}, namely
$$ \pas{\det\of X}^{s-1}\cdot\det\of{\matcominor{X}{I}{J}}.$$

\subsection{Accounting for the signs}
A necessary first step for this task is to investigate how the sign of a
permutation $\pi$ is changed by \EM{removing} a given partial permutation $\pi^\prime$:
We view this as \EM{erasing} all the edges belonging to $\pi^\prime$'s permutation diagram
$m_{\pi^\prime}$ from $\pi$'s permutation diagram $m_\pi$; see again \figref{fig:derivdet2s}.

\begin{lem}
\label{lem:remove-edge}
Let $\pi\in\symm_n$ be a permutation, and let $\pi^\star$ be the permutation
corresponding to the permutation diagram $m_\pi$ with edge $\pas{i,\pi\of i}$ removed. Then we have
$$
\sgn\of\pi = \pas{-1}^{\pi\of i -i}\cdot\sgn\of{\pi^\star}.
$$
\end{lem}
\begin{proof}
We count the 
the number of intersections 
with edge $\pas{i,\pi\of i}$ in $m_\pi$: Let $A=\range{i-1}$, $B=\range{n}\setminus\range{i}$,
$C=\range{n}\setminus\range{\pi\of i}$ and $D=\range{\pi\of i-1}$ (see \figref{fig:remove-edge} for
an illustration).

Assume $\card{\pi^{-1}\of C\cap A}=k$: Then edge $\pas{i,\pi\of i}$
clearly intersects the $k$ edges joining vertices from $A$ to vertices from $C$
(see again \figref{fig:remove-edge}).

The only other intersections with $\pas{i,\pi\of i}$
come from edges joining vertices from $B$ to vertices from $D$: Since $\pi$ is a bijection,
we have $\card{\pi^{-1}\of C\cap B} = n-\pi\of i - k$, whence
$\card{\pi^{-1}\of D\cap B}=\card{B\setminus\pi^{-1}\of C} = n-i-\pas{n-\pi\of i - k} = k+\pi\of i - i$.

Altogether, the removal of edge $\pas{i,\pi\of i}$ removes $2k+\pi\of i - i$ intersections of edges.
\end{proof}
\begin{figure}
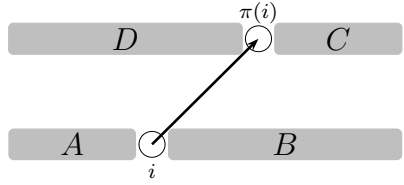

\caption{Erase a single edge $\pas{i,\pi\of i}$ in the permutation diagram $m_\pi$ of
some permutation $\pi$: Note that precisely the intersections with edges
leading from $A\subseteq\domain\of\pi$ to $C\subseteq\image\of\pi$
and with edges leading from $B\subseteq\domain\of\pi$ to $D\subseteq\image\of\pi$
are removed by this operation.
}
\label{fig:remove-edge}
\input graphics/remove-edge
\end{figure}

\begin{cor}
\label{cor:remove-edges}
Let $\pi\in\symm_n$ be a partitioned permutation $\pi=\pi_1\concat\pi_2$, where $\pi_1$ is the partial permutation
$$
\pi_1 =
\begin{pmatrix}\pi\of{i_1} & \pi\of{i_2} & \cdots & \pi\of{i_k} \\ {i_1} & {i_2} & \cdots & {i_k}\end{pmatrix}
$$
(with $\setof{i_1\leq i_2\leq\cdots\leq i_k}\subseteq\range{n}$). Clearly,
$\pi_2$ is the permutation
corresponding to the matching $m_\pi$ with edges $\pas{i_1,\pi\of{i_1}},\pas{i_2,\pi\of{i_2}},\dots,\pas{i_k,\pi\of{i_k}}$ erased, which we also denote by $\pi\setminus\pi_1$. Then we have
$$
\sgn\of\pi = \pas{-1}^{\sum_{j=1}^k\pi\of{i_k} -i_k}\cdot\sgn\of{\pi_1}\cdot\sgn\of{\pi_2}.
$$
If we denote $I=\setof{i_1,\dots,i_k}$ and $J=\setof{\pi\of{i_1}, \dots, \pi\of{i_k}}$, we may
rewrite this as
\begin{equation}
\label{eq:remove-edges}
\sgn\of\pi = \signsumset{I}{\range n}\cdot\signsumset{J}{\range n}\cdot\sgn\of{\pi_1}\cdot\sgn\of{\pi\setminus\pi_1}.
\end{equation}

\end{cor}
\begin{proof}
We proceed by induction: $k=1$ simply amounts to the statement of Lemma~\ref{lem:remove-edge}.

For $k>1$, let $i_{\text{max}}\defeq\pi^{-1}\of{\max\of{\image\pi_1}}$ be the pre-image of
the maximum of the image of $\pi_1$.
Let $l$ be the number of elements in the domain of $\pi_1$ wich are greater than $i_{\text{max}}$:
$$
l=\card{\setof{i:i\in\domain\pi_1 \and i>i_{\text{max}}}}.
$$
See \figref{fig:remove-edges} for an illustration. Removing the edge
$\pas{i_{\text{max}},\pi\of{i_{\text{max}}}}$ leaves (the diagram of) a permutation $\pi^\prime\in\symm_{n-1}$
and a partial permutation
$$
\pi_1^\prime =
\begin{pmatrix}
\pi\of{i_1} & \cdots & \pi\of{i_{\text{max}}-1} & \pi\of{i_{\text{max}}+1} & \cdots  & \pi\of{i_{k}-1} \\
{i_1} & \cdots & i_{\text{max}}-1 & \pas{i_{\text{max}}+1}-1 & \cdots & {i_k-1}
\end{pmatrix}
$$
therein of length $k-1$. By induction, we have
$$
\sgn\of{\pi^\prime} =
\pas{-1}^{\pas{\sum_{j=1}^k\pi\of{i_k} -i_k} - \pas{\pi\of{i_\text{max}}-i_\text{max}} - l}
\cdot\sgn\of{\pi^\prime\setminus\pi_1^\prime}\cdot\sgn\of{\pi_1^\prime}.
$$
Since we have
\bit
\item $\pi\setminus\pi_1 = \pi^\prime\setminus\pi_1^\prime\implies
\sgn\of{\pi_2}=\sgn\of{\pi\setminus\pi_1} = \sgn\of{\pi^\prime\setminus\pi_1^\prime}$,
\item $\sgn\of{\pi_1}=\pas{-1}^l\cdot\sgn\of{\pi_1^\prime}$ (see again \figref{fig:remove-edges}),
\item and $\sgn\of\pi=\pas{-1}^{\pi\of{i_\text{max}}-i_\text{max}}\sgn\of{\pi^\prime}$ (by
Lemma~\ref{lem:remove-edge}),
\eit
the assertion follows.
\end{proof}
\begin{figure}
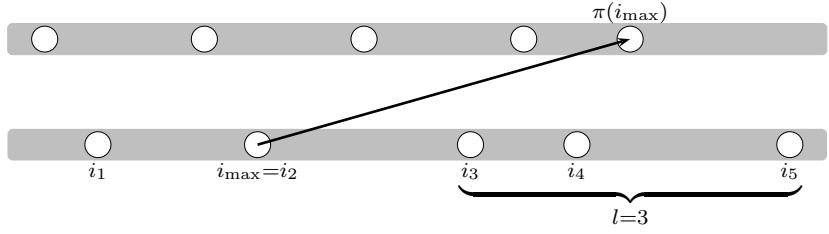

\caption{Erase several edges, corresponding to some partial permutation $\pi_1$  of $\pi$: Start with
the edge incident with the rightmost vertex from $\image\of\pi_1$ in $m_\pi$.
}
\label{fig:remove-edges}
\input graphics/remove-edges
\end{figure}

\subsection{Sums of (signed) products of minors}
Now consider a fixed equivalence class in the sense of \secref{sec:equivalence}, which is
indexed by a \EM{partition-scheme}
$$
\scheme{\tau_1\concat\tau_2\concat\cdots\concat\tau_s}.
$$
We want to compute the generating function $G_{\scheme{\tau}}$ of this equivalence class: Clearly, we may
concentrate on the \EM{nonempty} partial permutations; so w.l.o.g. we have to consider
the partition-scheme
$$
\scheme{\tau_1\concat\tau_2\concat\cdots\concat\tau_m}
$$
which consists only of \EM{nonempty} partial permutations $\tau_j$ for $1\leq j\leq m\leq s$.
For any $\sigma\in\symm_k$ with
$\sigma\complies\scheme{\tau_1\concat\tau_2\concat\cdots\concat\tau_m}$, such
partition scheme corresponds to a unique \EM{ordered partition} of the image of $\sigma$:
$$
\image\sigma=\range{k}=
\pas{\image\tau_1}\disjuni\pas{\image\tau_2}\disjuni\cdots\disjuni\pas{\image\tau_m}=J_1\disjuni J_2\disjuni\cdots\disjuni J_m,
$$
and any specification of a \EM{compatible ordered partition} $\compat{I}{J}=\pas{I_1,I_2,\dots,I_m}$, i.e.,
$$
\range{k}=I_1\disjuni I_2\disjuni\cdots\disjuni I_m\text{ where }\card{I_l}=\card{J_l},l=1,\dots,m,
$$
uniquely determines such $\sigma$, which we denote by $\sigma\of{\compat{I}{J},\scheme{\tau}}$.

Equation~\eqref{eq:remove-edges} gives the sign-change
caused by erasing the edges corresponding to $\tau_j$ (with respect to \EM{any} permutation in $\symm_n$
which contains $\tau_i$ as a partial permutation), whence we can write 
the generating function as
\begin{multline*}
G_{\scheme{\tau}} = \det\of X^{s-m} \\
\times
\sum_{\compat{I}{J}}
\sgn\of{\sigma\of{\compat{I}{J},\scheme{\tau}}}\cdot
\prod_{l=1}^m\pas{\sgn\of{\tau_l}\cdot
\signsumset{I_l}{\range{n}}\cdot\signsumset{J_l}{\range{n}}\cdot\det\of{\matcominor{X}{I_l}{J_l}}},
\end{multline*}
where the sum is over all compatible partitions $\compat{I}{J}$.
(The factor $\sgn\of{\sigma\of{\compat{I}{J},\scheme{\tau}}}$ comes from the \EM{determinant}
of partial derivatives.)
Clearly,
$$
\prod_{l=1}^m\signsumset{I_l}{\range{n}}=\prod_{l=1}^m\signsumset{J_l}{\range{n}}=1,
$$
so it remains to show
\begin{equation}
\label{eq:showthis}
\sum_{\compat{I}{J}}
\sgn{\sigma\of{\compat{I}{J},\scheme{\tau}}}\cdot
\prod_{l=1}^m\pas{\sgn\of{\tau_l}\cdot\det\of{\matcominor{X}{I_l}{J_l}}}
=
\det\of{X}^{m-1}\det\of{\matcominor{X}{\range k}{\range k}}.
\end{equation}
This, of course, is true for $m=1$. We proceed by induction on $m$.

For any ordered partition $S_1\disjuni S_2\disjuni\dots\disjuni S_m = \range{k}$,
we introduce the shorthand notation
$$\complementoffirst{S}{l}\defeq\range{k}\setminus\pas{S_1\disjuni S_2\disjuni\dots\disjuni S_l}.$$
Moreover, write $d_{I_j}\defeq\det\of{\matcominor{X}{I_j}{J_j}}$ for short. Then the lefthand-side
of \eqref{eq:showthis}
may be written as the $\pas{m-1}$-fold sum
\begin{equation}
\label{eq:multisum}
\sum_{\substack{I_1\subseteq\complementoffirst{I}{0}\\\card{I_1}=\card{J_1}}}\sgn\of{\tau_1}d_{I_1}
\!\!
\sum_{\substack{I_2\subseteq\complementoffirst{I}{1}\\\card{I_1}=\card{J_1}}}\sgn\of{\tau_2}d_{I_2}
\cdots
\!\!\!\!\!\!
\sum_{\substack{I_{m-1}\subseteq\complementoffirst{I}{m-2}\\\card{I_{m-2}}=\card{J_{m-2}}}}
\!\!\!\!\!\!
\sgn\of{\tau_{m-1}}d_{I_{m-1}}
\sgn\of{\tau_m}d_{I_m}\cdot\,\sgn\of{\sigma},
\end{equation}
where $I_m=\complementoffirst{I}{m-2}\setminus I_{m-1}$ and $\sigma=\sigma\of{\compat{I}{J},\scheme{\tau}}$.

Assume $\complementoffirst{J}{m-2}=\setof{j_1\leq\dots\leq j_a}$,
$\complementoffirst{I}{m-2}=\setof{i_1\leq\dots\leq i_a}$ and $J_m=\setof{j_{s_1}\leq\cdots\leq j_{s_b}}$.
Then the special choice $\overline{I_m}=\setof{i_{s_1}\leq\cdots\leq i_{s_b}}$ (i.e., with
respect to the relative ordering, ``$\overline{I_m}$ is the same subset as $J_m$ '')
and $\overline{I_{m-1}}=\complementoffirst{I}{m-2}\setminus \overline{I_m}$
determines \EM{uniquely} a partial permutation $\overline{\tau_{m-1}}$
$$
\overline{\tau_{m-1}}:\complementoffirst{I}{m-2}\to\complementoffirst{J}{m-2}.
$$
According to \eqref{eq:remove-edges}, by construction we have
\begin{equation}
\label{eq:aux}
\sgn\of{\overline{\tau_{m-1}}}=\sgn\of{\tau_{m-1}}\cdot\sgn\of{\tau_m}.
\end{equation}
Now consider $\sigma=\sigma\of{\compat{I}{J},\scheme{\tau}}$ in the innermost sum of
\eqref{eq:multisum}:
\EM{Erasing} the edges corresponding to $\tau_{m-1}$ and $\tau_{m-2}$ from $m_\sigma$ and
\EM{replacing} them by the edges corresponding to $\overline{\tau_{m-1}}$ yields a permutation
$\overline{\sigma}=\tau_1\concat\cdots\tau_{m-2}\concat\overline{\tau_{m-1}}$
(which, of course, complies to the partition scheme $\scheme{\overline{\tau}}=\scheme{\tau_1\concat\cdots\tau_{m-2}\concat\overline{\tau_{m-1}}}$).
Since by \eqref{eq:remove-edges} together with \eqref{eq:aux} we have
$$
\sgn\of{\overline{\tau_{m-1}}} = \sgn\of{\tau_{m-1}\concat\tau_m}\cdot\signsumset{{I_m}}{\complementoffirst{I}{m-2}}\cdot\signsumset{J_m}{\complementoffirst{J}{m-2}}
$$
and (clearly)
$$
\sigma\setminus\pas{\tau_{m-1}\concat\tau_m}=\overline{\sigma}\setminus\overline{\tau_{m-1}},
$$
we also have (again by \eqref{eq:remove-edges})
$$
\sgn\of\sigma=\signsumset{{I_m}}{\complementoffirst{I}{m-2}}\cdot\signsumset{J_m}{\complementoffirst{J}{m-2}}\cdot\sgn\of{\overline{\sigma}}.
$$
%
%
%
Hence the innermost sum of \eqref{eq:multisum} can be written as
$$
\sgn\of{\overline{\tau}}\cdot
\pas{
\sum_{I_{m-1}\subseteq\complementoffirst{I}{m-2}}
\!\!\!\!\!\!
\signsumset{I_m}{\complementoffirst{I}{m-2}}\cdot\signsumset{J_m}{\complementoffirst{J}{m-2}}
\cdot d_{I_{m-1}}
\cdot d_{I_m}
}\cdot\sgn\of{\overline\sigma}.$$
If we can show that this last sum equals
$\det\of X\cdot \det\of{\matcominor{X}{\complementoffirst{I}{m-2}}{\complementoffirst{J}{m-2}}}$, then
\eqref{eq:showthis} follows by induction, since the $\pas{m-1}$--fold sum in \eqref{eq:multisum}
thus reduces to
an $\pas{m-2}$--fold sum, which corresponds to the partition-scheme $\scheme{\overline{\tau}}=\scheme{\tau_1\concat\tau_2\concat\dots,\tau_{m-2}\concat\overline{\tau_{m-1}}}$.
\subsection{(A generalization of) Laplace's theorem}
Luckily, 
%
%
%
a generalization (see \cite[section 148]{Muir:1933}) of Laplace's Theorem
serves as the closer for our argumentation:
\begin{thm}
\label{thm:muir148}
Let $a$ be an $\pas{m+k}\times\pas{m+ k}$-matrix, and let
$1\leq i_1<i_2<\cdots<i_m\leq m+k$ and
$1\leq j_1<j_2<\cdots<j_m\leq m+k$ be (the indices of) $k$ fixed rows and
$k$ fixed columns of $a$. Denote the set of these (indices of) rows and columns by $R$ and $C$,
respectively.
Consider some fixed 
set $I\subseteq R$.
Then we have:
\begin{equation}
\label{eq:muir148}
\det\of{a}\cdot\det\of{\matcominor{a}{R}{C}} = 
\sum_{\substack{J\subseteq C,\\\card{J}=\card{I}}}
{\signsumset{I}{R}\cdot\signsumset{J}{C}}\cdot
\det\of{\matminor{a}{\overline{R\setminus I}}{\overline{C\setminus J}}}
\cdot
\det\of{\matminor{a}{\overline{I}}{\overline{J}}}
.
\end{equation}
\end{thm}
A combinatorial
proof for this identity is given in \cite[proof of Theorem 6]{Fulmek:2012}: So altogether,
we achieved a ``purely combinatorial'' proof for \eqref{eq:vivanti}. \hfill\qedsymbol

\bibliography{paper}

\begin{thebibliography}{1}

\bibitem{Aitken:1956}
A.C. Aitken.
\newblock {\em Determinants and Matrices}.
\newblock Oliver \& Boyd, Ltd., Edinburgh, 9th. edition, 1956.

\bibitem{Bergeron-Labelle-Leroux:1998}
F.~Bergeron, G.~Labelle, and P.~Leroux.
\newblock {\em Combinatorial Species and tree--like Structures}, volume~67 of
  {\em Encyclopedia of Mathematics and its Applications}.
\newblock Cambridge University Press, 1998.

\bibitem{Caracciolo:2013}
S.~Caracciolo, A.D. Sokal, and A.~Sportiello.
\newblock Algebraic/combinatorial proofs of {C}ayley-type identities for
  derivatives of determinants and pfaffians.
\newblock {\em Advances in Applied Mathematics}, 50(4):474 -- 594, 2013.

\bibitem{Fulmek:2012}
M.~Fulmek.
\newblock Viewing determinants as nonintersecting lattice paths yields
  classical determinantal identities bijectively.
\newblock {\em Electron. J. Combin.}, 19(3):P21, 2012.

\bibitem{Muir:1906}
T.~Muir.
\newblock {\em The Theory of Determinants in the historical order of
  development}, volume 4 volumes.
\newblock MacMillan and Co., Limited, London, 1906--1923.

\bibitem{Muir:1933}
T.~Muir.
\newblock {\em A Treatise on the Theory of Determinants}.
\newblock Longmans, Green and Co., London, 1933.

\bibitem{Vivanti:1890}
G.~Vivanti.
\newblock Alcune formole relative all'operazione {$\Omega$}.
\newblock {\em Rend. Circ. Mat. Palermo}, 1890.

\end{thebibliography}

\end{document}